\documentclass[11pt]{amsart}
\usepackage{Baskervaldx}
\usepackage[baskervaldx,cmintegrals,bigdelims,vvarbb]{newtxmath}
\usepackage[cal=cm]{mathalfa}
\usepackage{enumerate, geometry, pgfplots, graphicx, subcaption, amsmath, array, pdfsync, tikz, tikz-cd, color, url, float, wrapfig, comment, enumitem}

\usetikzlibrary{positioning}

\usepackage{colortbl}
\usepackage{tabulary}
\usepackage{etoolbox}
\usepackage{tabularx}
\usepackage{xfrac}
\newcolumntype{Y}{>{\centering\arraybackslash}X}
\usepackage{makecell}

\newcommand{\floor}[1]{\left\lfloor #1 \right\rfloor}

\usepackage[hyperfootnotes=false, pdfencoding=auto]{hyperref}
\hypersetup{
  colorlinks= true,
  linkcolor=[RGB]{0,89,42},
  urlcolor=[RGB]{0,89,42},
  citecolor=[RGB]{0,89,42}
}

\usepackage[font=footnotesize]{caption}

\usetikzlibrary{calc}

\numberwithin{equation}{section}

\title{The fibering genus of Fano hypersurfaces}
\author{Nathan Chen, Benjamin Church, Lena Ji, and David Stapleton}

\definecolor{mycolor}{RGB}{146, 214, 203}
\definecolor{myothercolor}{RGB}{179, 215, 232}

\newtheorem{theorem}{Theorem}
\newtheorem{corollary}[theorem]{Corollary}
\newtheorem{lemma}[theorem]{Lemma}
\newtheorem{proposition}[theorem]{Proposition}

\newtheorem{notation}[theorem]{Notation}
\numberwithin{theorem}{section}

\newtheorem{Lthm}{Theorem}

\newcommand{\hr}[2]{\hyperref[#1]{#2}}

\theoremstyle{definition}
\newtheorem{remark}[theorem]{Remark}

\newtheorem{construction}[theorem]{Construction}
\newtheorem{question}[theorem]{Question}
\newtheorem*{question*}{Question}

\newtheorem{example}[theorem]{Example}
\newtheorem{definition}[theorem]{Definition}

\def\CC{{\mathbb C}}

\def\PP{{\mathbb{P}}}

\def\Spec{{\mathrm{Spec} \ }}

\def\Oc{{\mathcal{O}}}

\def\Lc{{\mathcal{L}}}

\newcommand{\m}{\mathfrak{m}}
\newcommand{\iO}{\mathcal{O}}
\newcommand{\Csh}{\mathcal{C}}
\newcommand{\Supp}[2]{\mathrm{Supp}_{#1}\left(#2\right)}
\newcommand{\wt}[1]{\widetilde{#1}}

\newcommand{\embed}{\hookrightarrow}

\newcommand{\onto}{\twoheadrightarrow}

\def\Xc{{\mathcal{X}}}
\def\Ic{{\mathcal{I}}}

\def\cl{{\colon}}

\def\kapbar{{\overline{\kappa}}}

\DeclareMathOperator{\sign}{\mathrm{sign}}

\def\dra{{\dashrightarrow}}
\def\ra{{\rightarrow}}

\def\kapbar{{\overline{\kappa}}}

\def\Frac{{\mathrm{Frac} \ }}

\def\covgon{{\mathrm{cov.gon}}}
\def\fg{{\mathrm{fib.gen}}}
\def\fgon{{\mathrm{fib.gon}}}
\def\covgenus{{\mathrm{cov.gen}}}
\def\covgon{{\mathrm{cov.gon}}}
\def\irr{{\mathrm{irr}}}
\def\gon{{\mathrm{gon}}}

\renewcommand{\L}{\mathcal{L}}
\newcommand{\red}{\mathrm{red}}

\newcommand{\mybigwedge}{\raisebox{.25ex}{\scalebox{0.86}{$\bigwedge$}}}

\newcommand{\defi}[1]{\textit{#1}}
\newcommand{\kk}{\mathbf{k}}

\pgfplotsset{compat=1.15}
\definecolor{graycolor}{rgb}{0.66,0.66,0.66}

\newcommand{\Yeta}{\eta}
\newcommand{\Ynoteta}{\xi}

\thanks{During the preparation of this article, N.C. was partially supported by NSF MSPRF grant DMS-2103099, L.J. was partially supported by NSF MSPRF grant DMS-2202444, and D.S. was partially supported by NSF grant DMS-1952399.}


\begin{document}
\maketitle

\thispagestyle{empty}

\vspace{-24pt}
\begin{abstract}
Koll\'ar proved that a very general $n$-dimensional complex hypersurface of degree at least $3\lceil (n+3)/4\rceil$ is not birational to a fibration in rational curves. This is most interesting when the hypersurface is Fano, in which case it is covered by rational curves. In this paper, we extend Koll\'ar's ideas and show that for any genus \(g\), there are Fano hypersurfaces (in more restrictive degree and dimension ranges) that are not birational to fibrations in genus \(g\) curves. In other words, we show that the \defi{fibering genus} of these hypersurfaces can be arbitrarily large. The fibering genus of a variety has been studied in work of Konno, Ein--Lazarsfeld, and Voisin, but this is the first paper to explore these ideas in the Fano range. Following Koll\'ar, we degenerate to characteristic $p>0$ to rule out these fibrations. A crucial input is Tate's genus change formula and its generalizations, which imply that any regular curve of genus \(g\) is smooth if \(p\) is sufficiently large compared to \(g\).
\end{abstract}

\section{Introduction}

The purpose of this paper is to study which complex hypersurfaces \(X\) are birationally fibered in curves of a given genus, using the specialization to characteristic $p>0$ technique of Koll\'ar. In the Fano range, \(X\) is rationally connected, so it is always swept out by curves of genus \(0\). It is interesting to ask however, if $X$ can be swept out \textit{birationally} by curves of genus 0 (or higher genus for that matter). For very general Fano hypersurfaces $X\subset \PP^{n+1}_\CC$ of degree $d\ge 3\lceil \tfrac{n+3}{4} \rceil$, Koll\'ar proved that the answer is no for genus 0, namely \(X\) is not birational to a ruled variety or to a conic bundle (\cite{Kollar95} and \cite[Thm. V.5.14]{KollarRationalCurves}). In this paper, we extend Koll\'ar's results to higher genus.

We consider an invariant called the \defi{fibering genus} (Definition~\ref{def:fibgenus}), defined to be the smallest integer \(g\) such that $X$ admits a rational fibration in genus $g$ curves. For surfaces with trivial Albanese, this coincides with the minimal geometric genus of a pencil of connected divisors on $X$ and was first studied by Konno \cite{Konno08} in the context of smooth surfaces in $\PP^{3}$. Konno proved that in this case, pencils of minimal genus are given by projection from a line. Ein--Lazarsfeld investigated this invariant for polarized K3 surfaces and abelian surfaces \cite{EL20} of Picard rank 1 and showed that these invariants asymptotically grow like a constant times $\sqrt{L^{2}}$, where $L$ is a generator of the Picard group. Most recently, Voisin \cite{Voisin-fibrations} studied the fibering genus and other related invariants for hyper-K\"ahler manifolds. Our main result is that there are Fano hypersurfaces of arbitrarily large fibering genus.

\begin{Lthm}\label{thm:fibering-genus}
Let \(X\subset\PP^{n+1}_{\CC}\) be a very general hypersurface of degree \(d\) and dimension \(n \geq 3\). If
\[ d \geq p \left\lceil \frac{n + \floor{(g+5)/2}}{p+1}\right\rceil \]
for some positive integer $g \ge 1$ and prime \(p\geq 2g+3\), then 
\[ \fg(X)\geq g+1. \] 
For example, if \(d\geq 5\left\lceil\tfrac{n+3}{6}\right\rceil\), then \(X\) is not birational to a conic bundle or a genus one fibration.
\end{Lthm}

Every Fano variety is covered by genus \(0\) curves; however, as Theorem~\ref{thm:fibering-genus} shows, for every \(g\) there exists a Fano variety of fibering genus \( \geq g\) in sufficiently high dimension. In the case $g = 0$, our methods also recover Koll\'{a}r's theorem on conic bundle structures mentioned above (see Remark~\ref{rem:conic-bundles}). The prime \(p\) in Theorem~\ref{thm:fibering-genus} is chosen to ensure that any regular curve of arithmetic genus \(g\) is smooth in characteristic \(p\) (see \S\ref{sec:fibrations-bg}). In the table below, we record some values ensuring lower bounds on the fibering genus of a very general \(n\)-dimensional degree \(d\) hypersurface in \(\PP^{n+1}_{\CC}\) (asymptotically in \(n\)):

\begin{center}
    \setlength{\tabcolsep}{10pt}\def\arraystretch{1.25}
    \begin{tabular}{c | c c c c c c c}
    \(\fg(X)\geq\) & 1 & 2 & 3 & 5 & 6 & 8 & 9 \\
    \hline
    \(p\) & 3 & 5 & 7 & 11 & 13 & 17 & 19 \\
    \hline
    \(d \gtrsim \underline{\hspace{.5cm}} \) & \(3n/4\) & \(5n/6\) & \(7n/8\) & \(11n/12\) & \(13n/14\) & \(17n/18\) & \(19n/20\)
    \end{tabular}
\end{center}

We also obtain a lower bound for fibering genus that, although less sharp than Theorem~\ref{thm:fibering-genus}, is expressed explicitly in terms of \(n\) and \(d\) (Corollary~\ref{cor:explicit_bound}). One particular consequence of our bound is:

\begin{Lthm}[Follows from Corollary~\ref{cor:explicit_bound}]\label{thm:asymptotic-index-sqrt}
Let \(X\subset\PP^{n+1}_{\CC}\) be a very general hypersurface of dimension $n \ge 3$ and degree $d\ge n+2 - \tfrac{1}{4}\sqrt{n+2}$. Then
\[ \fg(X) \ge \frac{1}{5} \sqrt{n+2} - 1. \]
\end{Lthm}





\subsection{Comparison with previous results and other measures of irrationality}\label{sec:other-measures-irr} For \(g \leq 2\), there have been previous works showing the non-existence of genus \(g\) fibrations of hypersurfaces. The techniques of birational rigidity \cite{Pukhlikov-bir-rigid} can be used to rule out genus \(0\) fibrations in some cases (e.g. smooth index \(1\) Fano hypersurfaces, see \cite{Kollar19}). For \(g=1\), Grassi and Wen showed that in dimensions $3\leq n \leq 5$, smooth Calabi--Yau hypersurfaces do not admit genus \(1\) fibrations \cite{Grassi91, GrassiWen}. Assuming the existence and termination of klt flips, they also prove the same result for all $n \geq 3$. For genus \(1\) fibrations with a section (i.e. Jacobian fibrations of dimension one), the first and fourth author proved that very general hypersurfaces of degree $d\ge 5 \lceil \tfrac{n+3}{6} \rceil$ do not birationally admit this structure \cite{CS-rat-end}, by studying rational endomorphisms on hypersurfaces. Finally, complex varieties that have a birational fibration in genus 2 curves (or more generally hyperelliptic curves) admit a birational involution whose quotient is birational to the image of the relative canonical linear system of the fibration. One can rule out the existence of such birational involutions (and therefore, genus 2 fibrations) for very general Fano hypersurfaces of degree $d\ge 3\lceil \tfrac{n+3}{4}\rceil$ using \cite[Theorem A]{CJS22}.

Several other measures of irrationality involving curves have been previously defined and studied. For hypersurfaces, the covering gonality first appeared in work of Pirola--Lopez \cite{PL95}, where they showed that the \defi{covering gonality} of a smooth surface $X \subset \PP^{3}$ of degree $d \geq 4$ is equal to $\covgon(X) = d-2$. This was later generalized to higher-dimensional hypersurfaces of large degree by \cite{BDELU} and \cite{BCFS19}. 
Here \(\covgon(X)\) is the minimal gonality of the general fiber of a covering family \(\mathcal C\to X\) of curves. One can similarly define the \defi{covering genus} \(\covgenus(X)\) to be the minimal genus of such a covering family. We have the inequalities
\[\covgon(X)\leq \left\lfloor\frac{\covgenus(X)+3}{2}\right\rfloor\leq \left\lfloor\frac{\fg(X)+3}{2}\right\rfloor.\]
The results of \cite{BDELU, BCFS19} show that if $X_d \subset \PP^{n+1}_{\CC}$ is a general hypersurface of degree $d$, then $\covgon(X_d) \ge d - n$, thereby giving a lower bound on the fibering genus in the general type range (see Example~\ref{fg-general-type}). Our results do not quite reproduce this bound for $d \gg n$ (we do get a lower bound asymptotically linear in $d - n$ but with a worse constant), but we give an improvement in the range $d \lesssim n + \tfrac{1}{2} \sqrt{n}$. 

The first and fourth authors proved that if \(X\subset\PP^{n+1}_{\CC}\) is a very general hypersurface of degree \(d\geq n+1-\tfrac{1}{4}\sqrt{n+2}\), then \(\irr(X) \geq \tfrac{1}{4}\sqrt{n+2}\) \cite{CS-deg-irr}, where the \defi{degree of irrationality} \(\irr(X)\) is the minimal degree of a dominant rational map \(X\dashrightarrow\PP^n\) (in fact, they showed that the same bound holds for the minimal degree map to a ruled variety). However, there are no clear inequalities relating these invariants to the fibering genus. As we point out in Remark~\ref{rem:MapsToRuledVar}, the main result of \cite{CS-deg-irr} can be used to give a similar bound to Theorem~\ref{thm:asymptotic-index-sqrt} with $\frac{1}{5}$ replaced by $\frac{1}{8}$, as long as one can show by other methods that \(\fg(X)\neq 1\). In the present paper, we directly rule out fibrations by low genus curves, using an obstruction that applies uniformly to both genus $\le 1$ and $\ge 2$ fibrations.

For an arbitrary hypersurface $X \subset \PP^{n+1}_{\CC}$ of degree $d$, one can always project from a line $\ell$ and define a map $f_{\ell} \colon X \dashrightarrow \PP^{n-1}_{\CC}$. Depending on whether the line $\ell$ is contained in $X$ or not, the general fiber of $f_{\ell}$ will be either a degree $d-1$ or degree $d$ plane curve (possibly singular). Therefore, we always have the upper bounds $\fg(X) \leq \frac{1}{2}(d-1)(d-2)$ and $\fgon(X) \leq d-1$, where \(\fgon(X)\) is the minimal gonality of the general fiber of a fibration of \(X\) in curves.

\begin{question}
For Fano hypersurfaces over $\CC$, can one give any obstructions to the existence of low \textit{gonality} fibrations?
\end{question}

It is worth mentioning that there is no analogue of Theorem~\ref{thm:tate-genus-change} for gonality: one cannot guarantee smoothness of curves in large characteristic by fixing only the gonality (see Example~\ref{exmp:bad-curves}\eqref{item:non-normal}).

\subsection{Outline of argument} Theorem~\ref{thm:fibering-genus} is a consequence of Theorem~\ref{thm:fib_genus_inequality}, which we prove in three parts. First, we show that the fibering genus can only drop under specialization (Proposition~\ref{prop:fg-specializes}). Then we apply a construction of Mori to degenerate from a very general complex hypersurface to a $\mu_p$-cover of a hypersurface in characteristic $p>0$ (Construction~\ref{construction:degeneration}). Finally, for these $\mu_p$-covers, we give lower bounds on the fibering genus using Koll\'ar's construction of $(n-1)$-forms, thereby giving a lower bound on the fibering genus for complex hypersurfaces. One technical point is that Bertini's theorem fails in characteristic \(p\): it's possible for a morphism between smooth varieties to have everywhere singular fibers. In this last step, we apply Tate's genus change formula (and its generalizations, see Theorem~\ref{thm:tate-genus-change}) to ensure that these curve fibrations on the \(\mu_p\)-covers have smooth general fibers.

\subsection*{Acknowledgements} We would like to thank Johan de Jong, Louis Esser, Antonella Grassi, J\'anos Koll\'ar, Eric Riedl, and Ravi Vakil for enlightening conversations and discussions.

\section{Background}

By \defi{curve}, we mean a separated finite type scheme over $k$ of pure dimension $1$. A \defi{variety} \(X\) over \(k\) is an integral separated scheme of finite type, and we denote its function field by \(\kk(X)\). Note that by our convention, not every curve is a variety.

\begin{definition}
Let $C$ be a proper curve over $k$. The \defi{arithmetic genus of $X$ over \(k\)} is
\[ p_a(C/k) \coloneqq \dim_k H^1(C, \Oc_C). \]
\end{definition}
Note that our definition differs from another definition, \(\chi_k(\mathcal O_C) - 1\), which is frequently used in the literature. The two definitions coincide when \(H^0(C,\mathcal O_C)=k\) (e.g. if \(C\) is geometrically integral over \(k\)). In particular, if \(C\) is smooth, proper, and geometrically integral over \(k\), then \(p_a(C/k)\) is equal to the geometric genus \(\dim_k H^0(C,\omega_C)\).

In this paper, we will study the \defi{fibering genus} of hypersurfaces by specializing to characteristic $p$. For this reason, we give a definition that works over any field:

\begin{definition}\label{def:fibgenus}
Let $X$ be a proper variety over a field. The \defi{fibering genus}, or $\fg(X)$, is the smallest integer $g \geq 0$ such that there exists a normal proper model $\tilde{X}$ of $X$ and a morphism $\tilde{X} \rightarrow B$ to a variety $B$ of dimension equal to $\dim X - 1$ whose generic fiber is a geometrically irreducible curve of arithmetic genus $g$ over \(\kk(B)\).
\end{definition}
Note that the fibering genus is always a non-negative integer. In characteristic \(0\), this definition agrees with \cite[Definition 1.1]{Voisin-fibrations}.

\begin{definition}
Let $C$ be an integral curve over an algebraically closed field $k$. We define the gonality of $C$, \(\gon(C)\), to be the minimal degree of a map $\tilde{C} \rightarrow \PP^{1}_{k}$, where $\tilde{C}$ is the normalization of $C$.
\end{definition}

If \(C\) is a smooth curve over an algebraically closed field with \(p_a(C/k)=g\), then Brill--Noether theory shows that \cite[Proposition A.1(v)]{Poonen-gonality}\begin{equation}\label{eqn:gon-genus-inequality}\gon(C)\leq\left\lfloor\frac{g+3}{2}\right\rfloor.\end{equation} We will use this inequality to give lower bounds on the genus of curves that appear in fibrations.
\begin{example}\label{fg-general-type}
The inequality~\eqref{eqn:gon-genus-inequality} immediately implies \[\covgon(X) \leq \left\lfloor \frac{\covgenus(X)+3}{2}\right\rfloor\] (see \S\ref{sec:other-measures-irr} for the definitions of covering gonality and genus).
For smooth general type hypersurfaces, combining this inequality with \cite[Theorem A]{BDELU} and the inequality \(\covgenus(X)\leq\fg(X)\) shows that \textit{any} smooth hypersurface \(X\subset\PP^{n+1}_{\CC}\) of dimension \(n\) and degree \(d\geq n+2\) has \(\fg(X) \geq 2(d-n)-3\). However, we suspect that this bound is far from being optimal.
\end{example}

\subsection{Properties of fibrations by curves and their generic fibers}\label{sec:fibrations-bg}
In this section, we collect some results about curve fibrations that we will use in positive characteristic. First, recall that the properties of the geometric generic fiber reflect those of a general fiber.
\begin{lemma}\label{lem:geom-generic-fiber}
    Let \(f\colon X\to Y\) be a morphism of varieties over a perfect field.
    \begin{enumerate}
        \item\label{part:irred-fibers} If the generic fiber is geometrically irreducible over \(\kk(Y)\), then a general fiber is irreducible \cite[\href{https://stacks.math.columbia.edu/tag/0559}{Tag 0559}]{stacks-project}.
        \item\label{part:normal-etc-fibers} The generic fiber is geometrically normal (resp. smooth, geometrically reduced) over \(\kk(Y)\) if and only if a general fiber is normal (resp. regular, reduced) \cite[Proposition 2.1]{PW}.
    \end{enumerate}
\end{lemma}

The following result will allow us to ensure that the fibrations we use will have irreducible general fibers, by taking the Stein factorization and using Lemma~\ref{lem:geom-generic-fiber}\eqref{part:irred-fibers}.
\begin{lemma}[{\cite[Propositions 2.1 and 2.2]{Tanaka-invariants}}]\label{lem:fibrationa-irreducible}
Let \(V\) be a proper variety over a field \(k\). \begin{enumerate} \item If \(V\) is normal, then \(k\) is algebraically closed in \(\kk(V)\) if and only if \(H^0(V,\mathcal O_V)=k\). \item If \(k\) is algebraically closed in \(\kk(V)\), then \(V\) is geometrically irreducible over \(k\).  \end{enumerate}
\end{lemma}

In every prime characteristic \(p\), there exist fibrations by curves where the total space is normal but every fiber is singular. The generic fiber of such a fibration is a curve  over an imperfect field that is regular but \textit{not smooth}.
\begin{example}\label{exmp:bad-curves}
    Let \(k_0\) be a perfect field of characteristic \(p>0\).
    \begin{enumerate}
        \item If \(p=2\) or \(3\), consider a quasi-elliptic fibration over \(k_0\) \cite[\S1]{Bombieri-Mumford}, and let \(C\) be the generic fiber. Then \(C\) is regular but not smooth over \(k\coloneqq H^0(C,\Oc_C)\), and \(p_a(C/k)=1\).
        \item\label{item:non-normal} (Rosenlicht, \cite[6.9.3]{KambayashiMiyanishiTakeuchi}) Assume \(p\geq 3\), let \(k=k_0(s)\) be the function field in \(1\) variable, and let \(C\) be the normalization of the plane curve defined by \(y^2 z^{p-2} - x^p + s z^p\). Then \(H^0(C,\Oc_C)=k\), the arithmetic genus is \(p_a(C/k)=\frac{1}{2}(p-1)\), and the base change \(C\otimes_k k^{1/p}\) is non-normal.
        \item\label{item:non-reduced} Let \(k=k_0(s,t)\) be the function field in two variables, and let \(C\) be the plane curve defined by \(sx^p+ty^p+z^p\). Then \(C\) is a regular curve with \(H^0(C,\Oc_{C})=k\), the arithmetic genus is \(p_a(C/k)=\frac{1}{2}(p-1)(p-2)\), and the base change \(C\otimes_k k^{1/p}\) is non-reduced.
    \end{enumerate}
\end{example}

When the characteristic of the field is large compared to the genus, such pathological behavior as in Example~\ref{exmp:bad-curves} cannot occur. The following result generalizes Tate's genus change formula \cite{Tate}: 
\begin{theorem}[{\cite{Tate} in the geometrically integral case, \cite{PW} in general}]\label{thm:tate-genus-change}
    Let \(C\) be a regular integral proper curve over a field \(k\) of characteristic \(p>0\). Assume \(H^0(C,\mathcal O_C)=k\), and let \(g \coloneqq p_a(C/k)\). If \(p\geq 2g+3\), then \(C\) is smooth over \(k\).
\end{theorem}

\begin{proof}
    For \(g=0\), see \cite[Lemma 6.5]{CTX}. If \(g= 1\) and if \(p\geq 5\) then \(C\) is smooth over \(k\) by \cite[Corollary 1.8]{PW}, so we may assume \(g\geq 2\) and \(p\neq 2\). We will show that if \(C\) is not geometrically normal over \(k\), then \(p\leq 2g+1\). In this setting, let \(Y\) be the normalization of \((C\otimes_k k^{1/p})_{\red}\) and \(\phi\colon Y\to C\) the induced morphism. By \cite[Theorems 1.1 and 1.2]{PW} and \cite[Corollary 3.3]{NagamachiTakamatsu} we have \(\phi^*K_C - K_Y\sim(p-1)D\) for some nonzero effective Weil divisor \(D\). Taking \(\deg_k\) and using that \(\deg_k(\phi^*\omega_C) = \deg(\phi)\deg_k(\omega_C)\) \cite[\href{https://stacks.math.columbia.edu/tag/0AYZ}{Tag 0AYZ}]{stacks-project} and \(\deg_k(\omega_C) = 2g-2\) \cite[\href{https://stacks.math.columbia.edu/tag/0C19}{Tag 0C19}]{stacks-project}, we get
\begin{equation}\label{eqn:genus-bound-g} (p-1)[k^{1/p}:k] \leq (p-1)[k^{1/p}:k]\deg_{k^{1/p}}(\mathcal O_Y(D)) = \deg(\phi)(2g-2) - [k^{1/p}:k]\deg_{k^{1/p}}(\omega_Y). \end{equation}
By \cite[Lemma 5.4]{JW} and the fact that \(-\deg_{k^{1/p}}(\omega_Y) \leq 2\) by \cite[Lemma 6.5]{CTX}, we have that
\[\deg(\phi)(2g-2) - [k^{1/p}:k]\deg_{k^{1/p}}(\omega_Y) \leq [k^{1/p}:k] (2g-2 - \deg_{k^{1/p}}(\omega_Y))
\leq [k^{1/p}:k] 2g. \]
Thus, after dividing~\eqref{eqn:genus-bound-g} through by \([k^{1/p}:k]\), we get \(p-1 \leq (p-1)\deg_{k^{1/p}}(\mathcal O_Y(D)) \leq 2g.\)
\end{proof}
The bound in Theorem~\ref{thm:tate-genus-change} is sharp for all \(g\), as shown by Example~\ref{exmp:bad-curves}\eqref{item:non-reduced} with \(p=2\) for \(g=0\), and Example~\ref{exmp:bad-curves}\eqref{item:non-normal} for \(g\geq 1\) (noting that \(2g+2\) is never prime for \(g\geq 1\)).

\section{Degenerations of curves and fibering genus}

The main goal of this section is to prove that fibering genus specializes (Proposition~\ref{prop:fg-specializes}). In order to prove this result, we will first need a careful analysis of the numerical invariants of degenerations of curves with normal total space. Throughout, we will work in the following setting:

\begin{notation}\label{notation:DVR}
    Let $(R, \m)$ be an excellent DVR with fraction field $K = \Frac{R}$ and residue field $\kappa = R/\m$.
\end{notation}

\begin{definition}
A \defi{degeneration of curves} is a proper flat family $X \to \Spec{R}$ over a DVR $R$ where the generic fiber $X_{K}$ is an integral normal projective curve over $K$.
\end{definition}

Now let $X \to \Spec{R}$ be a normal degeneration of curves, and consider the following data. Let $\{\Gamma_i\}_{i=1}^r$ be the irreducible components of the reduced subscheme $(X_\kappa)_\red$ of the special fiber, and let \(\Gamma_i^\nu\) be the normalization of \(\Gamma_i\). Define the following $\kappa$-algebras:
\begin{enumerate}
\item $A = H^0(X_\kappa, \iO_{X_\kappa})$,
\item $\kappa' = H^0((X_\kappa)_\red, \iO_{(X_\kappa)_\red})$,
\item $\kappa_i = H^0(\Gamma_i, \iO_{\Gamma_i})$, and
\item \(\kappa_i^\nu = H^0(\Gamma_i^\nu, \iO_{\Gamma_i^\nu})\).
\end{enumerate}
Then $A$ is an Artin local $\kappa$-algebra, and $\kappa \subset \kappa' \subset \kappa_i$ and \(\kappa_i\subset\kappa_i^\nu\) are finite field extensions by \cite[\href{https://stacks.math.columbia.edu/tag/0BUG}{Tag 0BUG}]{stacks-project} and \cite[\href{https://stacks.math.columbia.edu/tag/04L2}{Tag 04L2}]{stacks-project}, since these schemes are connected and the last three are also reduced.

\begin{proposition} \label{prop:regular_inequality}
Let $X \to \Spec{R}$ be a regular degeneration of curves with $H^0(X_K,\Oc_{X_K})=K$. Let $\{ \Gamma_i \}_{i = 1}^r$ be the irreducible components of \((X_\kappa)_{\red}\) with reduced structure, and let \(\Gamma_i^\nu\) be the normalization of \(\Gamma_i\). Then 
\[ \sum_{i = 1}^r p_a(\Gamma_i^\nu / \kappa_i^\nu) \le \sum_{i = 1}^r p_a(\Gamma_i^\nu / \kappa_i) \le \sum_{i = 1}^r p_a(\Gamma_i / \kappa_i) \le p_a(X_K/K). \]
\end{proposition}

\begin{proof}
Let \(m_i\) be the multiplicity of \(\Gamma_i\) in \(X_\kappa\), let \(d \coloneqq \gcd(m_i)\), and define the effective Cartier divisor \(D \coloneqq \sum_{i = 1}^r (m_i / d) C_i\).
By \cite[\href{https://stacks.math.columbia.edu/tag/0C68}{Tag 0C68}]{stacks-project} and \cite[\href{https://stacks.math.columbia.edu/tag/0C69}{Tag 0C69}]{stacks-project}, $D$ satisfies
\begin{enumerate}
\item $\kappa_D \coloneqq H^0(D, \iO_{D})$ is a finite field extension of $\kappa$, and
\item $\chi(X_\kappa, \iO_{X_\kappa}) = d \cdot \chi(D, \iO_{D})$.
\end{enumerate}
Hence, if $g \coloneqq p_a(X_{K}/K)$ and $g_D \coloneqq p_a(D/\kappa_D)$, we have
\[ g - 1 = d [\kappa_D : \kappa] (g_D - 1). \]
Therefore,
\[  g_D = \frac{g-1}{d [\kappa_D : \kappa]} + 1 \le g \]
since either $g = 0$, in which case $g_D = 0$ and $d = [\kappa_D : \kappa] = 1$, or $g > 0$, in which case we see that $g_D \le g$. Furthermore, since $(X_\kappa)_{\red} = D_{\red}$, we see that $(X_\kappa)_{\red}$ is a proper $\kappa_D$-scheme. In particular $\kappa_D \subset \kappa'$ and also
\(H^1(X_\kappa, \iO_{D}) \onto H^1(X_\kappa, \iO_{(X_\kappa)_{\red}})\).
So we conclude that
\[ p_a((X_\kappa)_{\red}/\kappa') \le p_a((X_\kappa)_{\red}/\kappa_D) \le p_a(D/\kappa_D) = g_D \le g = p_a(X_K/K). \]
Next, we compare the arithmetic genus of \((X_\kappa)_{\red}\) to that of its components. Consider the finite map $\pi\colon \bigsqcup_{i=1}^r \Gamma_i \to (X_\kappa)_{\red}$ splitting the irreducible components. This gives a sequence of sheaves,
\begin{center}
\begin{tikzcd}
0 \arrow[r] & \iO_{(X_\kappa)_{\red}} \arrow[r] & \prod_{i=1}^r \iO_{\Gamma_i} \arrow[r] & \Csh \arrow[r] & 0,
\end{tikzcd}
\end{center}
and since $\dim{\Supp{}{\Csh}} = 0$, we have a surjection
\[ H^1\left((X_\kappa)_{\red}, \iO_{(X_\kappa)_{\red}}\right) \onto \bigoplus_{i=1}^r H^1(\Gamma_i, \iO_{\Gamma_i}). \]
Hence,
\[ \sum_{i = 1}^r p_a(\Gamma_i/\kappa_i) \le \sum_{i = 1}^r p_a(\Gamma_i / \kappa') \le p_a((X_\kappa)_{\red}/\kappa') \le p_a(X_K/K). \]
Finally, for each \(i\), we have the inequality \(p_a(\Gamma_i^\nu/\kappa_i) \leq p_a(\Gamma_i/\kappa_i)\) \cite[\href{https://stacks.math.columbia.edu/tag/0CE4}{Tag 0CE4}]{stacks-project}.
\end{proof}

\begin{proposition} \label{prop:normal_inequality}
Let $X \to \Spec{R}$ be a normal degeneration of curves over an excellent DVR, and assume $H^0(X_K,\Oc_{X_K}) = K$. Let $\Gamma_i \subset (X_\kappa)_{\red}$ be an irreducible component. Then
\( p_a(\Gamma_i^\nu / \kappa_i) \le p_a(X_K/K) \).
\end{proposition}

\begin{proof}
Since $R$ is excellent, by \cite{Lipman78}, there exists a strong desingularization $\pi \colon \wt{X} \to X$, meaning that \(\wt{X}\) is regular and \(\pi\) is an isomorphism over the regular locus of $X$. Then $\wt{X} \to \Spec R$ is a regular model of $\wt{X}_{K} \cong X_K$ and hence by Proposition~\ref{prop:regular_inequality} verifies the inequality. 
Since $\wt{X} \to X$ is an isomorphism away from a finite set of points, for each $\Gamma_i \subset X_\kappa$ there is an irreducible component $\wt{\Gamma}_i \subset \wt{X}_s$ mapping birationally onto $\Gamma_i$. 
This induces a birational morphism on the normalizations \(\wt{\Gamma}_i^\nu\to\Gamma_i^\nu\), which is an isomorphism.
\end{proof}

We now apply this to prove the following result, which we use in the proof of Proposition~\ref{prop:fg-specializes}.

\begin{corollary}\label{cor:SpecializingRelDim1DVR}
Let $f\colon X \to Y$ be a morphism of relative dimension $1$ between flat, proper, normal, integral schemes over an excellent DVR \(R\). Assume $Y_\kappa$ is integral and the map $f_\kappa \colon X_\kappa \to Y_\kappa$ is dominant.
Let $\Yeta \in Y$ and $\Ynoteta \in Y_\kappa$ be the generic points.
Let $X_1, \dots, X_r$ be the irreducible components of $(X_\kappa)_{\red}$, and define the multiplicities $m_i$ by
\[
X_\kappa = \sum_{i=1}^r m_i X_i.
\]
Let \(I\subset\{1,\ldots,r\}\) be the indices corresponding to the components \(X_i\) that dominate \(Y_\kappa\). For \(i\in I\), let $\Gamma_i$ be the generic fiber of $X_i\ra Y_\kappa$, and let \(\Gamma_i^\nu\to\Gamma_i\) be the normalization.

Then the following hold:
\begin{enumerate}
    \item\label{item:X_Ynoteta-decomposition} $X_\Ynoteta = \sum_{i\in I} m_i \Gamma_i$.
    \item\label{item:X_Ynoteta-genus} For $i\in I$, set $\kappa_i \coloneqq H^0(\Gamma_i, \Oc_{\Gamma_i})$ and \(\kappa_i^\nu \coloneqq H^0(\Gamma_i^\nu, \Oc_{\Gamma_i^\nu})\). Then 
    \[p_a(\Gamma_i^\nu/\kappa_i^\nu) \le p_a(\Gamma_i^\nu / \kappa_i) \le p_a(X_\Yeta / \kappa(\Yeta)).\]
    In particular, the Stein factorization \(X_i^\nu \to B_i \to Y_\kappa \) has the property that the general fiber of $X_i \to B_i$ is an irreducible curve, and the generic fiber \(C_i\) of \(X_i\to B_i\) has genus $p_a(C_i/\kappa_i^\nu)\le p_a(X_\Yeta/\kappa(\Yeta))$.
\end{enumerate}
\end{corollary}

\begin{proof}

Part~\eqref{item:X_Ynoteta-decomposition} follows from properties of localization and primary decomposition. For~\eqref{item:X_Ynoteta-genus}, first note that since \(Y\) is normal, $D \coloneqq \iO_{Y,\Ynoteta}$ is a DVR.
Now consider $X_D \to \Spec{D}$. Since \(X_D\) is a localization of the irreducible and normal scheme \(X\), it is also irreducible and normal. The morphism \(X_D\to\Spec D\) is proper by base change, and it is flat since it is dominant. Therefore, as $\Gamma_i^\nu$ and $C_i$ are isomorphic as curves over $\kappa_i^\nu\cong \kk(B_i)$, the genus inequalities in part~\eqref{item:X_Ynoteta-genus} follow from Proposition~\ref{prop:normal_inequality}. The statement about irreducibility of the general fibers of the Stein factorization follows from Lemma~\ref{lem:fibrationa-irreducible} and Lemma~\ref{lem:geom-generic-fiber}\eqref{part:irred-fibers}.
\end{proof}

Finally, we are ready to show that fibering genus can only drop under specialization.
We note that this property is known for covering gonality by \cite[Proposition 2.2]{GounelasKouvidakis}; however, their argument, which uses the Kontsevich moduli space, requires a base change in applying the valuative criterion of properness for stacks, and so it will not work for fibering genus.

\begin{proposition}\label{prop:fg-specializes}
Let $R$ be an excellent DVR as above, and let $\Xc$ be a normal proper integral scheme that is flat over \(R\).
\begin{enumerate}
\item\label{item:fg-specializes} If $\fg(\Xc_{K}) = g$, then for every component $\Xc_{\kappa}' \subset \Xc_{\kappa}$, we have $\fg(\Xc_{\kappa}') \leq g$.
\item\label{item:geometric-fg-specializes} Suppose $\Xc_{\kappa}'\subset \Xc_\kappa$ is a geometrically integral component of the central fiber that appears with multiplicity \(1\). If \(\Xc_K\) is geometrically integral with $\fg(\Xc_{\bar{K}}) = g$, then $\fg(\Xc_{\kapbar}') \leq g$.
\end{enumerate}
\end{proposition}

\begin{proof}
Set $\fg(\Xc_{K}) = g$, and let $f_{K} \colon \Xc_{K} \dashrightarrow B_{K}$ be a rational map computing the fibering genus of $\Xc_{K}$, i.e. the generic fiber of the normalization of the graph of $f_{K} \colon \Xc_K \dashrightarrow B_K$ is geometrically irreducible of arithmetic genus $g$. Let $B$ be a normal proper model of $B_K$ over $R$. The rational map $f_{K}$ extends to a rational map $f \colon \Xc \dashrightarrow B$ over $R$. Since $\Xc$ is normal, $f$ is defined on all codimension 1 points. By an argument of Abhyankar and Zariski \cite[Lemma 2.22]{KollarSingsMMP}, for any codimension 1 point $\delta \in \Xc$, there is a birational morphism $\mu \cl \tilde{B} \ra B$ so that for the induced rational map
\[
\tilde{f} \colon \Xc \dra \tilde{B},
\]
$\tilde{f}(\delta)$ is a codimension 1 regular point of $\tilde{B}$. Let $\delta$ now be the generic point of $\Xc_\kappa'$. Then the closure of the image $B_\kappa' \coloneqq \tilde{f}(\delta) \subset \tilde{B}_\kappa$ is a codimension 1 subscheme of $\tilde{B}$.

Now we use a standard argument to pass to morphisms. Consider the normalization of the closure of the graph of $\tilde{f}$, which we will denote by $\Gamma$. This admits a morphism
\[ \Gamma \rightarrow \Xc \times_{R} \tilde{B} \]
over $R$, with the projection maps $\pi_{1}, \pi_{2}$. Furthermore, the generic fiber of \(\Gamma\to B\) is the normalization of the graph of \(f_K\). The first projection $\pi_{1}$ is proper and is an isomorphism away from a codimension $\geq 2$ locus in $\Xc$. The central fiber $\Gamma_{\kappa}$ has pure codimension 1, and since $\pi_{1}$ is an isomorphism away from a codimension $\geq 2$ locus in $\Xc$, there is a unique component $\Gamma_\kappa'$ which dominates $\Xc_\kappa'$ via the first projection with degree one. Therefore, $\Gamma_{\kappa}'$ is birational to $\Xc_\kappa'$.

This defines a morphism $\Gamma \rightarrow \tilde{B}$ which restricts to $\Gamma_\kappa' \rightarrow B_\kappa'$ satisfying all of the conditions in Corollary~\ref{cor:SpecializingRelDim1DVR}. Applying Corollary~\ref{cor:SpecializingRelDim1DVR} to $\Gamma \rightarrow \tilde{B}$ over $R$ gives the desired inequality on fibering genera (noting that the normalization of the generic fiber of \(\Gamma'_\kappa\to B'_\kappa\) is the generic fiber of the morphism \({\Gamma'_\kappa}^\nu\to B'_\kappa\) from the normalization \cite[\href{https://stacks.math.columbia.edu/tag/0307}{Tag 0307}]{stacks-project}). This proves part~\eqref{item:fg-specializes}.

Part~\eqref{item:geometric-fg-specializes} follows from part~\eqref{item:fg-specializes} by base changing $\Xc$ by a (possibly ramified) finite extension \(\tilde{R}\) of the DVR $R$ and applying part~\eqref{item:fg-specializes} to the normalization \(\tilde{\Xc}\) of the base change, as in the proof of \cite[Theorem IV.1.6]{KollarRationalCurves}. The assumptions on \(\Xc_\kappa'\) are to ensure that \(\tilde{\Xc}_{\tilde{\kappa}}\) has a component birational to \((\Xc_\kappa')_{\tilde{\kappa}}\).
\end{proof}

\section{Lower bounds on fibering genus for Fano hypersurfaces}

The main result of this section is Theorem~\ref{thm:fib_genus_inequality}, which will imply Theorem~\ref{thm:fibering-genus} and Corollary~\ref{cor:explicit_bound}.

\begin{definition}
    Let \(X\) be a variety over an algebraically closed field \(k\), and let \(\L\) be a line bundle on \(X\). We say that sections of \(\L\) \defi{separate \(\gamma\) points on an open set \(U\subset X\)} if for any \(\gamma\) distinct points \(x_1,\ldots,x_\gamma\in U\), there is a section \(s\in H^0(X,\L)\) that vanishes on \(x_1,\ldots,x_{\gamma-1}\) but not on \(x_\gamma\).
\end{definition}

This is slightly different from the notion of \((\mathrm{BVA})_p\) as in \cite[Definition 1.1]{BDELU}. For example, a line bundle \(\Lc\) is \((\mathrm{BVA})_1\) if and only if it separates 2 points and tangent vectors on an open set. For our purposes, we only require our line bundle to separate \(\gamma\) distinct points.

\begin{lemma}\label{lem:SepPointsFibGon} Fix an integer $\gamma \geq 2$, and let $f \colon X \to B$ a proper morphism of relative dimension $1$ of normal varieties over an algebraically closed field $k$, and assume the generic fiber is a smooth geometrically connected curve over \(\kk(B)\). Let $n = \dim{X}$. Suppose there is a line bundle $\L$ embedding in $\Omega_X^{k}$ for some $1\leq k \leq n$
\[
\varphi : \L \embed \Omega_X^{k}
\] 
whose sections separate $\gamma$ points on an open set. Then the general fiber of $f$ has gonality $\geq \gamma + 1$.
\end{lemma}

\begin{proof}
Since \(X\) is normal and \(f\) has relative dimension \(1\), its singular locus does not dominate \(B\), i.e. a general fiber of \(f\) is contained in the regular locus of \(X\). Thus,
we can shrink $B$ so that $f$, $X$, and $B$ are all smooth, and any fiber \(C\) is a proper, smooth, irreducible curve. Fix some fiber $C$ that meets the open set $U$ on which $\L$ separates $\gamma$ points.
Then the conormal sequence
\[ 0 \to \Ic_C/\Ic_C^2 \to \Omega_{X}|_C \to \Omega_C \to 0 \]
is an exact sequence of vector bundles, and since $\Omega_C$ is a line bundle, there is an induced exact sequence of vector bundles
\[ 0 \to \mybigwedge^{k}(\Ic_C/\Ic_C^2) \to \Omega^{k}_X|_C \to (\mybigwedge^{k-1} (\Ic_C/\Ic_C^2)) \otimes \Omega_C \to 0. \]
However, since $C$ is a fiber of $f$ we have $\Ic_C/\Ic_C^2 \cong \Oc_{C}^{n-1}$ and hence $\mybigwedge^{k-1}(\Ic_C / \Ic_C^2) \cong \Oc_{C}^{{n-1 \choose k-1}}$. Therefore, we get ${n - 1 \choose k-1}$ projection maps:
\[ \L |_{C} \to \Omega_X^{k}|_C \to \Omega_C. \]
The projection maps are all zero exactly if $\L |_{C} \embed \Omega_X^{k} |_{C}$ factors through $\mybigwedge^{k} (\Ic_C/\Ic_C^2) \cong \Oc_{C}^{{n-1 \choose k}}$ (when $k = n$, this just says that the projection map is nonzero). In this case, we could only have $h^{0}(C, \L |_{C}) \not= 0$ if $\L |_{C} \cong \Oc_{C}$, but that would imply that sections of $\L$ cannot separate more than one point on $C$, contradicting the assumption that $\gamma \ge 2$. Therefore, one of the projections $\L|_{C} \to \Omega_C$ must be a nonzero map of line bundles and hence injective, meaning that
\[ H^0(C, \L |_{C}) \to H^0(C, \Omega_C) \]
must be injective. Since we chose $C$ to be general, $H^0(C, \L |_{C})$ and hence $H^0(C, \Omega_C)$ can separate $\gamma$ points in some open $(U \cap C) \subseteq C$. Therefore $\gon(C) \geq \gamma + 1$ by geometric Riemann--Roch (c.f. \cite[Lemma 1.3]{BDELU}).
\end{proof}

In the proof of Theorem~\ref{thm:fib_genus_inequality}, we will use a construction of Mori to degenerate a very general hypersurface to a \(\mu_p\)-cover. We recall several key facts that Koll\'ar showed about this family:

\begin{construction}[{\cite{Kollar95}, \cite[Construction V.5.14.4]{KollarRationalCurves}}]\label{construction:degeneration}
Let \(p\) be a prime, and let \(n,e\) be positive integers with \(n\geq 3\). Let \(R\) be a mixed characteristic DVR with algebraically closed residue field \(\kappa\) of characteristic \(p\). There is a normal proper integral scheme \(\Xc\to \Spec R\) with the following properties:
\begin{enumerate}
    \item The generic fiber \(\Xc_K\) is a hypersurface of degree \(pe\) in \(\PP^{n+1}_K\).
    \item The special fiber is a \(\mu_p\)-cover \(\Xc_\kappa \to Y_e\) of a smooth degree \(e\) hypersurface in \(\PP^{n+1}_\kappa\). Moreover, \(\Xc_\kappa\) admits a resolution of singularities \cite[paragraph 21]{Kollar95}
    \[\begin{tikzcd} \Xc'_\kappa \arrow[r]\arrow[rr,bend left=-30,swap,"r"] & \Xc_\kappa \arrow[r] & Y_e, \end{tikzcd} \]
    and the line bundle $L\coloneqq{r}^{\ast}\Oc_{\PP^{n+1}_{\kappa}}(pe+e-n-2)$ injects into $\mybigwedge^{n-1} \Omega_{\Xc'_\kappa}$ \cite[line (15.3)]{Kollar95}.
\end{enumerate}
\end{construction}

We now prove the main result: a lower bound on fibering genus.

\begin{theorem}\label{thm:fib_genus_inequality}
Let $X_{n,d} \subset \PP^{n+1}_{\CC}$ be a very general hypersurface of degree $d$ and dimension \(n\geq 3\). For any prime $p$ and positive integer $e$ such that $pe \le d$, define the quantity $\gamma \coloneqq pe+e-n-1$. If $\gamma\ge 2$, then
\[ \fg(X_{n,d}) \ge \min \left\{ \frac{p-2}{2}, 2\gamma -1 \right\}. \]
\end{theorem}

\begin{proof}
First we consider the case \(d=pe\). By Construction~\ref{construction:degeneration}, a very general degree \(pe \) hypersurface \(X_{n,pe}\subset\PP^{n+1}_{\CC}\) admits a degeneration to \(\Xc_\kappa\) with a resolution of singularities \(\Xc'_\kappa\to\Xc_\kappa\). The line bundle $L = {r}^{\ast}\Oc_{\PP^{n+1}_{\kappa}}(\gamma-1)$ injects into $\mybigwedge^{n-1} \Omega_{\Xc'_\kappa}$ and separates $\gamma$ points on an open set of $\Xc'_\kappa$. If $\Xc'_\kappa$ birationally admits a fibration in curves such that the general fiber \(C\) is smooth, then Lemma~\ref{lem:SepPointsFibGon} (applied to the normalization of the graph of the rational fibration of $\Xc'_\kappa$ and \(k=n-1\)) implies that it has gonality $\gon(C) \geq \gamma + 1$. Here we use the assumption $\gamma \geq 2$ here.

We will now obstruct low genus fibrations of $X_{n,pe}$. Assume for contradiction that \(X_{n,pe}\) birationally admits a fibration by irreducible genus \( g \) curves for some $g < \frac{p-2}{2}$. Using the degeneration of $X_{n,d}$ to $\Xc_\kappa$ in characteristic \(p\) (Construction~\ref{construction:degeneration}), we apply Proposition~\ref{prop:fg-specializes} to conclude that \(\Xc_\kappa\) birationally admits a fibration whose general fibers are irreducible curves of arithmetic genus \( g' \leq g \). The condition that \( p \geq 2g + 3 \geq 2 g'+3 \) implies that the general fiber \(C\) is smooth by Theorem~\ref{thm:tate-genus-change} and Lemma~\ref{lem:geom-generic-fiber}\eqref{part:normal-etc-fibers}. Therefore,
\begin{equation}\label{eqn:gonality-inequality-in-proof} \gamma + 1 \leq \gon(C) \leq \floor{\frac{g+3}{2}}, \end{equation}
and hence,
\[ g \ge 2(\gamma + 1) - 3 = 2\gamma - 1 \]
proving the claim in the case \(d=pe\).

For \(d>pe\), we degenerate a very general \(X_{n,d}\) to the union of a very general \(X_{n,pe}\) and \(d-pe\) hyperplanes. Proposition~\ref{prop:fg-specializes} then implies that \[\fg(X_{n,d}) \geq \fg(X_{n,pe}) \ge \min \left\{ \frac{p-2}{2}, 2\gamma -1 \right\}.\]
\end{proof}

Figure~\ref{fig:fg} gives a graphical representation of the lower bound from Theorem~\ref{thm:fib_genus_inequality}.

\begin{proof}[Proof of Theorem~\ref{thm:fibering-genus}]

Define \(r \coloneqq \lfloor\tfrac{g+3}{2}\rfloor\) and \( e \coloneqq \lceil \tfrac{n + r + 1}{p+1}\rceil \). We have $r \ge 2$ by the assumption that $g \geq 1$. Then
\[ \gamma = e (p+1) - n - 1 \geq n + r + 1 - n - 1 = r. \]
Suppose for the sake of contradiction that a very general \(X_{n,pe}\) admits a genus $g' \le g$ fibration. Then the proof of Theorem~\ref{thm:fib_genus_inequality} shows that $\Xc_\kappa$, as defined in Construction~\ref{construction:degeneration}, birationally admits a fibration whose general fiber is a smooth (here we use $p \geq 2g+3$) irreducible curve $C$ of genus at most $g'$, and hence $C$ satisfies the inequalities
\[ r+1 \leq \gamma + 1 \leq \gon(C) \leq \floor{\frac{g'+3}{2}} \leq \floor{\frac{g+3}{2}}. \]
This contradicts our choice of $r$. Thus, \( \fg(X_{n,pe}) \geq g + 1 \). The case \(d > pe\) follows from the inequality \(\fg(X_{n,d})\geq\fg(X_{n,pe})\), as in the proof of Theorem~\ref{thm:fib_genus_inequality}.
\end{proof}

\begin{remark}\label{rem:conic-bundles}
    For \(g=0\), taking \(p=3\) and \(r=2\) (instead of \(\lfloor\tfrac{g+3}{2}\rfloor\)) in the proof of Theorem~\ref{thm:fibering-genus} recovers Koll\'ar's bound ruling out (birational) conic bundle structures in degrees $d \geq 3\lceil \tfrac{n+3}{4}\rceil$.
\end{remark}

\begin{figure}[ht]
\centering
\includegraphics[width=0.8\textwidth]{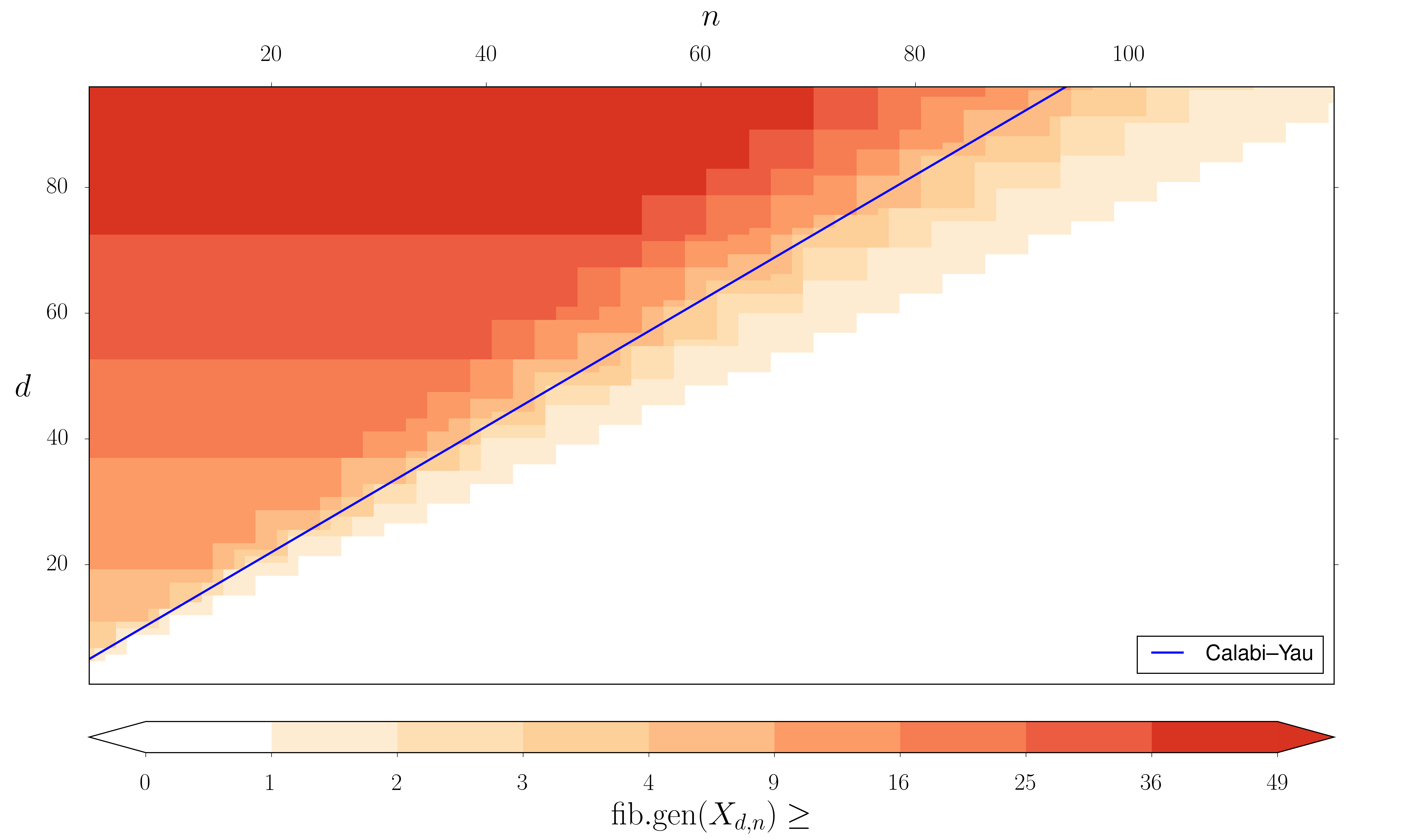}
\caption{Lower bounds on \(\fg(X_{n,d})\) for very general hypersurfaces \(X_{n,d}\subset\PP^{n+1}_{\CC}\) of degree \(d\).}\label{fig:fg}
\end{figure}

\section{Explicit bounds on the fibering genus in terms of the dimension and degree}

In this section, we obtain an explicit lower bound for the fibering genus in terms of degree and dimension, as another consequence of Theorem~\ref{thm:fib_genus_inequality}.

\begin{corollary}\label{cor:explicit_bound}
Let \(X\subset\PP^{n+1}_{\CC}\) be a very general hypersurface of degree \(d\) and dimension \(n \geq 3\). Then
\[ \fg(X) \ge \frac{-\iota + \sqrt{\iota^2 + (9/2) d}}{9} - 1 \qquad \text{ where } \quad \iota = n+2 - d. \]
\end{corollary}
\noindent Note that if \(\iota>0\), then \(X\) is Fano and $\iota$ is the \defi{Fano index} of $X$.

The idea to prove Corollary~\ref{cor:explicit_bound} is to use Bertram's postulate to find a simultaneous lower bound, in terms of \(n\) and \(d\), for both of the quantities in Theorem~\ref{thm:fib_genus_inequality}. This bound is less sharp than Theorem~\ref{thm:fibering-genus} because for specific pairs $(n, d)$, a more optimal prime $p$ can often be chosen to give better bounds from a rough estimate using Bertrand's postulate.

\begin{proof}
    Let \(\theta\) be the unique positive real number satisfying the quadratic equation
    \[ \frac{(\theta/2) - 2}{2} = 2 \left( d - \theta + \frac{d}{\theta} - 1 - n \right) - 3,\] 
    which can explicitly be written as
    \[ \theta = \frac{-\iota + \sqrt{\iota^2 + (9/2) d}}{9/4}. \]
    We need to show that the fibering genus is at least $\frac{\theta}{4}-1$.
    There is nothing to show if $\theta < 2$, so we may assume that $\theta \ge 2$. By Bertrand's postulate, we may choose some prime $p$ satisfying $\tfrac{1}{2} \theta \le p \le \theta$. The first inequality $\tfrac{1}{2} \theta \le p$ implies $\frac{p-2}{2} \geq \frac{(\theta/2) -2}{2}$. Furthermore, if we set $e \coloneqq \floor{\frac{d}{p}}$, then clearly $e \ge \frac{d}{p} - 1$, so the second inequality $p \le \theta$ implies
    \[ 2(pe + e - n) - 3 \ge 2\left(d - p + \frac{d}{p} - 1 - n \right) - 3 \ge 2 \left(d - \theta + \frac{d}{\theta} - 1 - n\right) - 3 = \frac{(\theta/2) - 2}{2}  \]
    Therefore, by Theorem~\ref{thm:fib_genus_inequality}, we have
    \[ \fg(X) \ge \min \left \{ \frac{p-2}{2}, 2(pe + e - n) - 3 \right\} \ge \frac{\theta}{4} - 1. \]
\end{proof}

For Calabi--Yau hypersurfaces, i.e. the case that $d = n+ 2$, Corollary~\ref{cor:explicit_bound} shows:
\begin{corollary}\label{cor:CY-asymptotic}
    Let \(X_{n+2}\subset\PP^{n+1}_{\CC}\) be a very general hypersurface of degree \(n+2\) and dimension \(n\geq 3\). Then \[\fg(X_{n+2}) \geq \frac{1}{3 \sqrt{2}} \sqrt{n + 2} - 1.\]
\end{corollary}
Thus, asymptotically the fibering genus of Calabi--Yau hypersurfaces grows by a factor of \(\sqrt{n+2}\). We emphasize that Theorem~\ref{thm:fib_genus_inequality} gives sharper bounds than Corollary~\ref{cor:CY-asymptotic}.

\begin{remark}
    As mentioned in the introduction, Grassi and Grassi--Wen showed that smooth Calabi--Yau hypersurfaces do not birationally admit genus \(1\) fibrations (conditionally on standard MMP conjectures if \(n\geq 6\)) \cite{Grassi91,GrassiWen}. Our methods recover their bound for very general Calabi--Yau hypersurfaces for \(n=3\), \(n=5\), and \(n\geq 8\) (unconditionally on the MMP).
\end{remark}

Next, for any \(d\), applying Jensen's inequality to Corollary~\ref{cor:explicit_bound} yields the following bound for a very general hypersurface \(X\subset\PP^{n+1}_{\CC}\) of degree $d$ and dimension $n \ge 3$:
\[ \fg(X) \ge \frac{1 + 2^{-1/2} \sign(d-n-2)}{9} \cdot (d - n - 2) + \sqrt{\frac{d}{36}} - 1. \]




This result illustrates the expected behavior of $\fg(X)$. In the $d \gg n$ regime, there is a dominant term linear in $d-n$, and in the regime $d \le n$, the dominant behavior is $\sqrt{d}$. However, the constants that appear in this estimate are certainly not optimal.

Finally, we would first like to point out how earlier work of the first author and fourth author can be used to rule out fibrations in curves of geometric genus $g \geq 2$ for very general complex Fano hypersurfaces in a certain range. Recall that, using the methods in the present paper, one can obtain a similar bound with an improved constant and without having to separately rule out the possibility that \(\fg(X)=1\) (Theorem~\ref{thm:asymptotic-index-sqrt}).

\begin{remark}\label{rem:MapsToRuledVar}
Let $X \subset \PP^{n+1}_{\CC}$ be a very general hypersurface of degree $d$. If $d > n+1-\tfrac{1}{4}\sqrt{n+2}$ and $X \dashrightarrow B$ is a fibration in curves of geometric genus $g \geq 2$, then one can use the main result of \cite{CS-deg-irr} to show that $g \geq 1 + \tfrac{1}{8}\sqrt{n+2}$. Indeed, given a rational fibration $X \dashrightarrow B$ in curves of geometric genus $g \geq 2$, after resolving the map and resolving the source, we have a morphism $\pi \colon \tilde{X} \rightarrow B$ whose general fiber is a smooth curve of genus $g$. Consider the composition of maps
\[ X \dashrightarrow \PP \pi_{\ast}\omega_{\tilde{x}/B} \simeq_{\text{bir.}} \PP^{g-1}_{B} \dashrightarrow \PP^{1}_{B} \]
where the last map is given by projecting away from $g-2$ general rational sections. This map has degree $2g-2$. On the other hand, by \cite{CS-deg-irr} we know that the minimal degree map from $X$ to a ruled variety is $\geq \tfrac{1}{4}\sqrt{n+2}$. So
\[ 2g-2 \geq \frac{1}{4}\sqrt{n+2} \implies g \geq 1 + \frac{1}{8}\sqrt{n+2}. \]
\end{remark}

\bibliographystyle{alpha}
\bibliography{references}

\footnotesize{
\textsc{Department of Mathematics, Columbia University, New York 10027} \\
\indent \textit{E-mail address:} \href{mailto:nathanchen@math.columbia.edu}{nathanchen@math.columbia.edu}

\textsc{Department of Mathematics, Stanford University, California 94305} \\
\indent \textit{E-mail address:} \href{mailto:bvchurch@stanford.edu}{bvchurch@stanford.edu}

\textsc{Department of Mathematics, University of Michigan, Ann Arbor, Michigan 48109} \\
\indent \textit{E-mail address:} \href{mailto:lenaji.math@gmail.com}{lenaji.math@gmail.com}

\textsc{Department of Mathematics, University of Michigan, Ann Arbor, Michigan 48109} \\
\indent \textit{E-mail address:} \href{mailto:dajost@umich.edu}{dajost@umich.edu}
}

\end{document}